\documentclass[11pt]{amsart}

\usepackage{amsmath}

\usepackage{amsfonts}

\usepackage{amssymb}

\numberwithin{equation}{section}

\newtheorem{lem}{Lemma}[section]

\newtheorem{teo}[lem]{Theorem}

\newtheorem{defi}[lem]{Definition}

\newtheorem{oss}[lem]{Remark}

\renewcommand{\phi}{\varphi}

\newcommand{\R}{{\mathbb{R}}}

\newcommand{\bbR}{{\mathbb{R}}}

\newcommand{\esssup}{{\mathrm{esssup}}}

\newcommand{\average}{{\mathchoice {\kern1ex\vcenter{\hrule height.4pt

width 6pt

depth0pt} \kern-9.7pt} {\kern1ex\vcenter{\hrule height.4pt width 4.3pt

depth0pt}

\kern-7pt} {} {} }}

\newcommand{\eps}{\epsilon}

\title[Rigidity results in Diffusion Markov Triples]{
Rigidity results in Diffusion Markov Triples}

\thanks{ 
The authors are members of {\em Gruppo Nazionale per l'Analisi
Ma\-te\-ma\-ti\-ca, la Probabilit\`a e le loro Applicazioni} (GNAMPA)
of the {\em Istituto Nazionale di Alta Matematica} (INdAM). Supported by
the Australian Research Council Discovery Project Grant ``N.E.W. Nonlocal Equations
at Work''.}

\author{Serena Dipierro, Andrea Pinamonti, Enrico Valdinoci}

\address{Dipartimento di Matematica ``Federigo Enriques'',
Universit\`a degli studi di Milano,
Via Saldini 50, 20133 Milano (Italy).}

\email{serena.dipierro@unimi.it}

\address{Dipartimento di Matematica, Universit\`a di Trento,
Via Sommarive 14, 38123 Povo, Trento (Italy).}

\email{andrea.pinamonti@gmail.com}

\address{Dipartimento di Matematica ``Federigo Enriques'',
Universit\`a degli studi di Milano,
Via Saldini 50, 20133 Milano (Italy), and
School of Mathematics and Statistics,
University of Melbourne, Grattan Street, 
Parkville, VIC-3010 Melbourne (Australia), and
Istituto di Matematica Applicata e Tecnologie Informatiche,
Via Ferrata 1, 27100 Pavia (Italy).
}

\email{enrico.valdinoci@unimi.it}

\begin{document}

\begin{abstract}
We consider stable solutions of semilinear
equations in a very general setting.
The equation is set on a Polish topological space
endowed with a measure and the linear operator is induced by
a carr\'e du champs (equivalently, the equation is set in a
diffusion Markov triple).

Under suitable curvature dimension conditions, we establish that
stable solutions with integrable carr\'e du champs are necessarily constant
(weaker conditions characterize the structure
of the carr\'e du champs and 
carr\'e du champ it\'er\'e).

The proofs are based on a geometric Poincar\'e formula in this setting.
{F}rom the general theorems established, several previous
results are obtained as particular cases and new ones are provided as well.
\end{abstract}

\maketitle

\section{Introduction}  
A quasilinear equation is an expression of the form
\begin{equation}\label{EQUAFP}
Lu+F(u)=0.
\end{equation}
Here, $L$ is a linear operator and identity~\eqref{EQUAFP}
often turns out to provide a significant constraint for the solution~$u$:
namely, at any point~$x$, the operator~$Lu$ at~$x$ has to be perfectly
balanced by the nonlinear source~$F(u(x))$ and, as a consequence,
the operator is constant along the level sets of the solution.

It is conceivable that this rigid constraint implies suitable classification
results: for this, one has typically to consider problems
in which~\eqref{EQUAFP} arises from a variational structure
and focus on solutions with
``sufficiently small energy'', since high energy solutions may develop
some ``wild behavior''.
To this end, one often considers solutions which are local minimizers
of the energy functional. Nevertheless, this minimal property may be
uneasy to check in practice and it is therefore customary to look at
a more general class of solutions, the so called ``stable'' solutions,
for which the second derivative of the energy functional is nonnegative
(in this setting, local minimizers become a special subclass of
the stable solutions). We refer to the monograph~\cite{DUPAI}
for a throughout discussion on stable solutions and on several classification
results.
\medskip

After \cite{FarHab, FSV},
a very useful tool towards the classification
of stable solutions
has been provided by a series of geometric Poincar\'e inequalities,
originally introduced by
Sternberg and Zumbrun in the celebrated articles~\cite{SZ1} and~\cite{SZ2}.
Roughly speaking, in this approach
a weighted $L^2$-norm of any test function
is controlled by a weighted $L^2$ norm of its
gradient. The advantage of this method
is that the weights are nonnegative
and possess a geometric interpretation,
hence the possible vanishing of the integral
in the Poincar\'e-type inequalities
implies the vanishing of the corresponding geometric weight, which
in turn provides a series of useful geometric rigidities.
Rigidity results via Poincar\'e-type inequalities
have been recently obtained in different settings,
including, among the others,
systems of equations~\cite{DI, fazly, FAZSI1},
manifolds~\cite{fsv1, fsv2, FMV, punzo1, FAZ17, DPV-VAR},
stratified groups~\cite{FV1, PV, BIRI10, FP},
equations with drift~\cite{FNP},
stratified media~\cite{SAVV, CHE, DP2}
and fractional equations~\cite{FRAC1, FRAC2, DP1},
and there are also applications for
equations in infinite dimensional spaces~\cite{CNV}.
The method can be also applied to deduce new
weighted Poincar\'e inequalities from
the explicit knowledge of a stable solution,
see~\cite{INDI}, and it is also flexible enough to deal with Neumann
boundary conditions~\cite{punzo2, DPV}.
\medskip

In this framework, a special role is often played by the geometry
of the ambient space. To understand this phenomenon, one can
think about the case of one-dimensional solutions in the Euclidean flat space
for a bistable nonlinearity, for instance heteroclinic solutions
of the mechanical pendulum.
If one wants to ``bend'' these objects to construct solutions on a sphere,
a geometric difficulty arises from the fact that the asymptotics ``at the
point infinity'' is not well-determined, thus making it difficult to perform
such a bending operation. This very heuristic example suggests
that for ``curved'' manifolds the number and the structure
of the stable solutions could be very different from the flat case.\medskip

The main objective of this paper is to deal
with stable solutions of semilinear equations in the very general
setting provided by the Markov triples, see the monograph~\cite{BLG}.
Though we have not attempted to reach the widest generality,
the setting of
Markov triples is an excellent
setting that comprises several particular cases at once,
by developing an appropriate form of calculus
and often providing
a general, elegant and unified treatment.
Moreover, the setting of Markov triples finds applications
in probability and mathematical physics, e.g. to describe
quantum ensembles, see e.g.~\cite{CHENG} and the references therein.
\medskip

The environment provided by the Markov triple (described here in details in
Section~\ref{89237ef2S}) is that of a Polish space endowed by a measure
and a carr\'e du champ operator, which provide a variational framework
for a general form of equation~\eqref{EQUAFP}.
See e.g.~\cite{HIR, BOU, BH} for a classical introduction to these
kinds of Dirichlet forms and also~\cite{AmbTre}
and the references therein for recent developments.\medskip

The main result of this paper (stated in details in Theorem~\ref{main1})
is that stable solutions with integral carr\'e du champ
in environments with suitable geometric assumptions
satisfy additional rigidity properties. Roughly speaking,
under strong curvature assumptions, these solutions are necessarily constant
and under weak curvature assumptions their level sets need to satisfy suitable
geometric identities.

These types of results will be obtained here as a consequence
of a very general Poincar\'e-type inequality (given explicitly
in Theorem~\ref{eqnlin}).\medskip

The rest of this paper is organized as follows.
In Section~\ref{89237ef2S}, we introduce in details
the setting of Markov triples in which we work and we state
our main results, the proofs
of which are given in Section~\ref{0pqwdlfjvsw3hb}.

Then, in Section~\ref{oqidwuwiefyetytyyt},
we deduce several particular cases from our main results
(some of these results were already known in the literature,
but follow here as a byproduct of our general and unified approach;
some other results seem to be new to the best of our knowledge).

\section{Functional setting and main results}\label{89237ef2S}

A triple  $(X,\mu,\Gamma)$ is called a \emph{diffusion Markov triple} if it
is composed by a Polish topological space $(X,\tau)$ (i.e. a separable completely
metrizable topological space~$X$ with topology~$\tau$) endowed with a $\sigma-$finite Borel measure $\mu$ with full support, a class $\mathcal{A}_0$ of real-valued measurable functions on $X$ and a symmetric bilinear map (the \emph{carr\'e du champ}) such that
$\Gamma:\mathcal{A}_0\times\mathcal{A}_0\to \mathcal{A}_0$ satisfying the following conditions:

\begin{enumerate}
	\item  $\mathcal{A}_0\subset L^{1}(\mu)$ is a vector space, dense in every $L^p(\mu)$, $1\leq p<\infty$, such that
	\[
	\forall f,g\in\mathcal{A}_0\Longrightarrow fg\in \mathcal{A}_0
	\]
	and
	\[
	\forall f_1,\ldots , f_k\in \mathcal{A}_0,\ \Psi\in C^{\infty}(\mathbb{R}^k),\ \Psi(0)=0\Longrightarrow \Psi(f_1,\ldots, f_k)\in \mathcal{A}_0,
	\]
	\item The map $\Gamma$ is bilinear, symmetric and such that, for any $f\in\mathcal{A}_0$, it holds that
$$\Gamma(f):=\Gamma(f,f)\geq 0.$$
Moreover, for any
$f_1,\ldots , f_k,g\in \mathcal{A}_0$, and any smooth $C^{\infty}$ function $\Psi : \mathbb{R}^k\to \mathbb{R}$ such that $\Psi(0)=0$, we have that~$
\Psi(f_1,\ldots, f_k)\in\mathcal{A}_0$ and
	\[
	\Gamma(\Psi(f_1,\ldots, f_k), g)=\sum_{i=1}^k \partial_i\Psi(f_1,\ldots, f_k)\Gamma(f_i,g)\qquad \mu-\mbox{a.e.}
	\]
	\item For every $f\in \mathcal{A}_0$, there is $C=C(f)>0$ such that for every $g\in\mathcal{A}_0$
	\[
	\left| \int_X \Gamma(f,g)\, d\mu\right| \leq C\|g\|_{L^2(\mu)}.
	\]
	The Dirichlet form $\mathcal{E}$ is defined for every $(f,g)\in\mathcal{A}_0\times\mathcal{A}_0$ by
	\[
	\mathcal{E}(f,g):=\int_X \Gamma(f,g)\, d\mu.
	\]
	The domain $\mathcal{D}(\mathcal{E})$ of $\mathcal{E}$ is the completion of $\mathcal{A}_0$ with respect to the norm 
$$\|f\|_{\mathcal{E}}:=(\|f\|_{L^2(\mu)}^2+\mathcal{E}(f,f))^{1/2}.$$ The Dirichlet form $\mathcal{E}$ is extended to $\mathcal{D}(\mathcal{E})$ by continuity together with the map $\Gamma$.
	\item $L$ is a linear operator on $\mathcal{A}_0$ defined as
	\[
	\int_X g Lf\, d\mu=-\int_X \Gamma(f,g)\, d\mu
	\]
	for all $f,g\in\mathcal{A}_0$. The domain of the operator $L$, $\mathcal{D}(L)$, is defined as the set of $f\in \mathcal{D}(\mathcal{E})$ for which there exists a constant $C=C(L)>0$ such that for any $g\in \mathcal{D}(\mathcal{E})$
	\[
	|\mathcal{E}(f,g)|\leq C\|g\|_{L^2(\mu)}.
	\]
	On $\mathcal{D}(L)$, the operator $L$ is extended via integration by parts formula for every $g\in\mathcal{D}(\mathcal{E})$. The operator $L$ defined on $\mathcal{D}(L)$ is always self-adjoint.
	\item There exists an increasing sequence $(\xi_k)_{k\geq 1}\subset \mathcal{A}_0$
	of functions such that~$\xi_k(x)\in[0,1]$ for any~$x\in X$, and
	\begin{equation*}
	\lim_{k\to+\infty} \xi_k=\mathrm{1}\qquad \mu-\mbox{a.e. in $X$}
	\end{equation*}
	and
	\begin{equation}\label{xi2}
	\Gamma(\xi_k)\leq \frac{1}{k}\qquad k\geq 1.
	\end{equation}
\end{enumerate}
We also assume the existence of an algebra $\mathcal{A}_0\subset \mathcal{A}$ of measurable functions
on $E$, containing the constant functions and satisfying the following requirements:
\begin{enumerate}
	\item[(1')] Whenever $f\in\mathcal{A}$ and $h\in\mathcal{A}_0$, $hf\in\mathcal{A}_0$;
	\item[(2')] For any $f\in\mathcal{A}$, if 
\[\int_E hf d\mu \geq 0\qquad \mbox{for every positive}\qquad h\in \mathcal{A}_0,\]
 then $f\geq 0,$\\
\item[(3')] $\mathcal{A}$ is stable under composition with smooth $C^{\infty}$-functions $\Psi: \mathbb{R}^k\to \mathbb{R}$ with~$\Psi(0)=0$,
namely if~$f_1,\dots,f_k\in\mathcal{A}$ then~$\Psi(f_1,\dots,f_k)\in
\mathcal{A}$. Furthermore,
for all~$f_1,\dots,f_k,g\in\mathcal{A}$,
it holds that
\begin{equation}\label{9090}
\Gamma(\Psi(f_1,\ldots, f_k), g)=\sum_{i=1}^k \partial_i\Psi(f_1,\ldots, f_k)\Gamma(f_i,g)\qquad \mu-\mbox{a.e.}
\end{equation}
\\
\item[(4')] The operator $L : \mathcal{A}\to \mathcal{A}$ is an extension of $L$ on $\mathcal{A}_0$. The carr\'e du champ
operator $\Gamma$ is also defined on $\mathcal{A}\times\mathcal{A}$ by the formula, for every $(f, g) \in \mathcal{A}\times\mathcal{A}$,
\begin{equation}\label{f per g}
\Gamma(f,g)=\frac{1}{2}\left[L(fg)-fLg-gLf\right]\in\mathcal{A},
\end{equation}
and for any~$f\in\mathcal{A}$ we set~$\Gamma(f):=\Gamma(f,f)$,
\item[(5')] For every $f\in\mathcal{A}$, $\Gamma(f,f)\geq 0$,\\
\item[(6')] For every $f\in\mathcal{A}$ and $g\in\mathcal{A}_0$, the integration by part formula
\[
\int_X \Gamma(f,g)\, d\mu=-\int_X g Lf d\mu=-\int_X f Lg\, d\mu
\]
holds true.
\item[(7')] If $f\in\mathcal{A}$ is such
that $\Gamma(f)=0$ then $f$ is constant.
\end{enumerate}

\begin{oss} {\rm
The assumption on the sign of~$\Gamma(f)$ in~$(5')$
and the nondegeneracy
condition in~$(7')$
are important and nontrivial structural
assumptions.
Roughly speaking, they reflect the ``ellipticity’’ of the operator~$L$.
For instance, if, for~$\R^2\ni(x,y)\mapsto f(x,y)$,
one considers the d'Alembert
operator~$
Lf=f_{xx}-f_{yy}$, one obtains
$$ \Gamma(f) = \frac12\,L(f^2) - f Lf
= ( (f_x)^2 +  f f_{xx} ) - ( (f_y)^2 +  f f_{yy} ) - f ( f_{xx}-f_{yy} )
=  (f_x)^2 - (f_y)^2 ,$$
which has indefinite sign.

In addition, if, for~$\R\ni x\mapsto f(x)$,
one considers the derivative operator~$Lf = f_x$, then it follows that
$$ \Gamma(f) = \frac12\,L(f^2) - f Lf
=\frac12\,(f^2)_x - ff_x = 0,$$
therefore there are nontrivial operators producing carr\'e du champs
which vanish identically.

It is interesting to point out that operators in ``divergence''
and ``nondivergence'' form share the same
carr\'e du champs. For instance, if~$a_{ij}$ is a smooth symmetric
matrix, and
\begin{equation}\label{DIVandNONDIV} L_{D}f:= \sum_{i,j=1}^n (a_{ij} f_i)_j
\quad{\mbox{ and }}\quad
L_{ND}f:= \sum_{i,j=1}^n a_{ij} f_{ij},\end{equation}
a direct computation shows that both~$L_D$ and~$L_{ND}$
satisfy
$$ \Gamma(f,g)=\sum_{i,j=1} a_{ij} \,f_i\,g_j.$$
Remarkably, the difference between the operators~$L_D$ and~$L_{ND}$
is read, in our setting, by condition~(6'), which
is satisfied by~$L_D$ and not by~$L_{ND}$. That is, in a sense,
while
conditions~(5') and~(7') reflect an elliptic condition
into a positive definiteness of an associated quadratic form,
condition~(6') detects the variational structure of the associated operator.
}
\end{oss}

\begin{oss}{\rm
Taking $\Psi(x,y)=xy$ in \eqref{9090}
we find that
\begin{equation}\label{9191}
\Gamma(f_1f_2, g)= f_1 \Gamma(f_2, g) + f_2\Gamma(f_1, g)\qquad \mu-\mbox{a.e.}
\end{equation}
for any $f_1,f_2, g\in \mathcal{A}$.

In addition, from~\eqref{9191}, we infer that, for any~$f,g\in
\mathcal{A}$,
\begin{equation}\label{9292}
\Gamma(f^2, g^2)= 
2f\Gamma(f, g^2)=
2f\Gamma(g^2,f)=4fg
\Gamma(g,f)=4fg
\Gamma(f,g)
\qquad \mu-\mbox{a.e.}
\end{equation}

Moreover,
exploiting~\eqref{9191} with~$f_1=f_2=g=1$, we see that
$$ \Gamma(1)= \Gamma(1) + \Gamma(1)\qquad \mu-\mbox{a.e.}$$
and therefore
\begin{equation*}
\Gamma(1)=0\qquad \mu-\mbox{a.e.}\end{equation*}
Using this identity and formula~\eqref{f per g} with~$f=g=1$, we also infer that
$$ 0=\Gamma(1)=\frac{1}{2}\left[L(1)-L(1)-L(1)\right]
=-\frac{L(1)}{2},$$
and so
\begin{equation}\label{and sp}
L(1)=0\qquad \mu-\mbox{a.e.}\end{equation}
}\end{oss}

To detect the behavior of second derivative operators,
it is also classical to introduce the following notation.

\begin{defi}\label{defGamma2}
The \emph{carr\'e du champ it\'er\'e} is the bilinear form $\Gamma_2:\mathcal{A}\times\mathcal{A}\to \mathcal{A}$ defined as
\begin{equation}\label{GammaD2}
\Gamma_2(f,g):=\frac{1}{2}\left[L\Gamma(f,g)-\Gamma(f, Lg)-\Gamma(g,Lf)\right].
\end{equation}
We define $\Gamma_2(f):=\Gamma_2(f,f)$.
\end{defi}

As an example, we point out that when~$L$ is the Laplace operator
in~$\R^n$, then the carr\'e du champ it\'er\'e reduces to the square
of the norm of the Hessian matrix
(see also Appendix~C.5 in~\cite{BLG}
for more general formulas
for Riemannian manifolds).

Now we recall a classical notion of
curvature dimension condition in our setting:

\begin{defi}\label{DEF2.3}
We say that $(X,\mu,\Gamma)$ satisfies the $CD(K,\infty)$ condition, for some $K\in\mathbb{R}$, if for any $f\in\mathcal{A}$
\[
\Gamma_2(f)\geq K\Gamma(f)
.\]
\end{defi}
The following result will be crucial
in our setting (see page~\pageref{QUI}). For the proof
of it, see \cite[formula~(3.3.6)]{BLG}.

\begin{teo}\label{QUI0}
Assume that~$(X,\mu,\Gamma)$ satisfies the $CD(K,\infty)$
condition in Definition~\ref{DEF2.3}, for some~$K\in\mathbb{R}$. Then
\begin{align}\label{Dis}
4\Gamma(f)\left(\Gamma_2(f)-K\Gamma(f)\right)\geq \Gamma(\Gamma(f))
\end{align}
for every $f\in \mathcal{A}$.
\end{teo}
In this paper we study the solutions to the following boundary value problem:
\begin{align}\label{eqncomp}
L u+F(u)=0\quad  \mbox{in}\ X,
\end{align} 
where~$F\in C^{\infty}(\bbR)$.
As customary, we say that~$u$ is a weak solution to~\eqref{eqncomp}
if~$u\in \mathcal{A}$ and
\begin{align}\label{weak}
\int_{X}\Gamma(u,\varphi)\, d\mu =\int_{X} F(u)\varphi \, d\mu,
\quad {\mbox{ for any }} \varphi\in \mathcal{A}_0.
\end{align}
Moreover,
we say that a weak solution~$u$ is stable if
\begin{align}\label{hstab}
\int_{X}\Gamma(\varphi) \, d\mu- \int_{X} F'(u)\varphi^2\, d\mu\ge 0, \qquad {\mbox{ for any }}
\varphi\in \mathcal{A}_0.
\end{align} 

\begin{teo}\label{main1}
Assume that~$(X,\mu,\Gamma)$ satisfies the curvature dimension condition $CD(K,\infty)$ 
for some $K\geq 0$ and $\Gamma:\mathcal{A}\times \mathcal{A}\to \mathcal{A}$.
For any~$(x,y)\in X\times X$, let 
\begin{equation}\label{DIS}
d(x,y):=\esssup\{f(x)-f(y)\}
\end{equation}
be the distance function in $(X,\mu,\Gamma)$, where the
essential supremum is computed on bounded functions
$f\in\mathcal{A}$ with $\Gamma(f)\leq 1$.
Let $x_0\in X$ and~$\rho(x):=d(x,x_0)$, for any~$x\in X$, and
suppose that there exists a sequence of functions~$\rho_k\in{\mathcal{A}}$
such that 
\begin{equation}\label{DIS2}
\rho_k\to\rho \qquad\mu-\mbox{a.e. in }X\quad{\mbox{ and }}\quad
\|\Gamma(\rho_k)\|_{L^{\infty}}\leq C_0,\end{equation}
for some~$C_0>0$.

Let~$u\in \mathcal{A}$ be a stable solution
to~\eqref{eqncomp} with
\begin{equation}\label{mai us}
\int_{X}\Gamma(u)\, d\mu <\infty
.\end{equation}
Then:
\begin{align}
&K>0\Longrightarrow \Gamma(u)=0\qquad\mu-\mbox{a.e. in }X;\\
\label{KPOSI0}
&K=0 \Longrightarrow \Gamma_2(u)-\Gamma\left(\Gamma(u)^{\frac{1}{2}}\right)=0\qquad\mu-\mbox{a.e. in }X.
\end{align}
In addition, if~$K>0$ and $\Gamma(u)\in \mathcal{A}\cap C^0(X)$ then
\begin{align}\label{KPOSI}
{\mbox{$u$ is constant in~$X$}}.
\end{align}
\end{teo}

The distance function in~\eqref{DIS}, often called intrinsic distance,
has been considered also in~\cite{BLG} and it coincides
with the Riemannian distance if~$X$ is a Riemannian manifold
and with the Carnot-Carath\'eodory distance if~$X$ is a
Carnot-Carath\'eodory space
(see~\cite{BLU} for the definition).
In this setting,
Assumption~\eqref{DIS2} is related
to the fact that~$\rho$ is Lipschitz
as defined in~\cite[Definition 3.3.24]{BLG}. 
We recall that the same assumption appears in \cite{Sturm} and it 
is the analogous to $|\nabla\rho|\leq 1$ which is satisfied by geodesic
distances on any manifold. 

The proof of Theorem~\ref{main1} is
based on a geometric Poincar\'e-type inequality, which we state
in this setting as follows:

\begin{teo}\label{eqnlin}
Let~$u\in\mathcal{A}$ be stable weak solution to~\eqref{eqncomp}. Then,
\begin{equation}\label{GF}
\int_{X}\left(\Gamma_2(u)-\Gamma\left(\Gamma(u)^{\frac{1}{2}}\right)\right)\varphi^2\, d\mu\leq \int_{X} \Gamma(u)\Gamma(\varphi)
\end{equation}
for any~$\varphi\in \mathcal{A}_0$.
\end{teo}

\section{Proof of Theorems~\ref{main1} and~\ref{eqnlin}}\label{0pqwdlfjvsw3hb}

\begin{proof}[Proof of Theorem~\ref{eqnlin}:]
We fix $\varepsilon>0$ and take~$\varphi\in \mathcal{A}_0$. By $(4')$, 
we know that~$\Gamma(u)\in\mathcal{A}$. Since~$\mathcal{A}$ is an algebra
(and therefore a vector space) it follows that
\begin{equation}\label{02}
\Gamma(u)+\varepsilon\in\mathcal{A}.\end{equation}
Now, we consider a function~$\Psi\in C^\infty(\R)$ such that~$\Psi(r)=0$ for
any~$r\le\varepsilon/4$ and~$\Psi(r)=\sqrt{r}$ for any~$r\ge\varepsilon/2$.
In view of~$(3')$ and~\eqref{02}, we have that
\begin{equation*}
\sqrt{\Gamma(u)+\varepsilon}=\Psi\big(\Gamma(u)+\varepsilon\big)
\in\mathcal{A}.\end{equation*}
Consequently,
by~$(1')$ we conclude that $$\psi_{\varepsilon}:=\left(\sqrt{\Gamma(u)+\varepsilon}\right)\varphi\in \mathcal{A}_0.$$
Hence, applying~\eqref{hstab} with~$\varphi$
replaced by~$\psi_{\varepsilon}$, we get
\begin{equation}\label{9393}
0\leq \int_{X} \Gamma\left(\left(\sqrt{\Gamma(u)+\varepsilon}\right)\varphi\right)\, d\mu-\int_{X}F'(u)\left(\Gamma(u)+\varepsilon\right)\varphi^2\, d\mu.
\end{equation}
Furthermore, by~\eqref{9191}, we have that
\begin{equation*}
\begin{split}&
\Gamma\left(\left(\sqrt{\Gamma(u)+\varepsilon}\right)\varphi\right)\\
=\;&
\Gamma\left(\left(\sqrt{\Gamma(u)+\varepsilon}\right)\varphi,\,
\left(\sqrt{\Gamma(u)+\varepsilon}\right)\varphi\right)\\
=\;&
\sqrt{\Gamma(u)+\varepsilon}\,
\Gamma\left(\varphi,\,
\left(\sqrt{\Gamma(u)+\varepsilon}\right)\varphi\right)
+\varphi
\Gamma\left( \sqrt{\Gamma(u)+\varepsilon},\,
\left(\sqrt{\Gamma(u)+\varepsilon}\right)\varphi\right)
\\ =\;&
\sqrt{\Gamma(u)+\varepsilon}\,
\Gamma\left(
\left(\sqrt{\Gamma(u)+\varepsilon}\right)\varphi, \varphi\right)
+\varphi
\Gamma\left( 
\left(\sqrt{\Gamma(u)+\varepsilon}\right)\varphi,\,
\sqrt{\Gamma(u)+\varepsilon}\right)
\\ =\;&
\big(\Gamma(u)+\varepsilon\big)\,
\Gamma\left(\varphi, \varphi\right)
+
\left(\sqrt{\Gamma(u)+\varepsilon}\right)\varphi\,
\Gamma\left(
\sqrt{\Gamma(u)+\varepsilon}, \varphi\right)
\\&\quad+\left(\sqrt{\Gamma(u)+\varepsilon}\right)\varphi\,
\Gamma\left( 
\varphi,\,
\sqrt{\Gamma(u)+\varepsilon}\right)
+\varphi^2
\Gamma\left( \sqrt{\Gamma(u)+\varepsilon},\,
\sqrt{\Gamma(u)+\varepsilon}\right)\\
=\;&
\big(\Gamma(u)+\varepsilon\big)\,
\Gamma(\varphi)
+
2\left(\sqrt{\Gamma(u)+\varepsilon}\right)\varphi\,
\Gamma\left(
\sqrt{\Gamma(u)+\varepsilon}, \varphi\right)
+\varphi^2
\Gamma\left( \sqrt{\Gamma(u)+\varepsilon}\right).
\end{split}\end{equation*}
This and~\eqref{9292} imply that
$$
\Gamma\left(\left(\sqrt{\Gamma(u)+\varepsilon}\right)\varphi\right)=
\big(\Gamma(u)+\varepsilon\big)\,
\Gamma(\varphi)
+\frac12\,
\Gamma\left(
\Gamma(u)+\varepsilon,\, \varphi^2\right)
+\varphi^2
\Gamma\left( \sqrt{\Gamma(u)+\varepsilon}\right).
$$
Plugging this information into~\eqref{9393}, we obtain that
\begin{equation}\label{9494}\begin{split}&
\int_{X}F'(u)(\Gamma(u)+\varepsilon)\varphi^2\, d\mu\\
\leq\;& \int_{X} \left(\Gamma(u)+\varepsilon\right)\Gamma(\varphi)+ \frac{1}{2}\,\Gamma(\varphi^2, \Gamma(u)+\varepsilon)+\varphi^2 \Gamma\left(\sqrt{\Gamma(u)+\varepsilon}\right)\, d\mu
.\end{split}\end{equation}
Now, we remark that, by Fatou's Lemma,
\begin{equation}\label{9595}
\liminf_{\varepsilon\to0}
\int_{X}F'(u)(\Gamma(u)+\varepsilon)\varphi^2\, d\mu\ge
\int_{X}F'(u) \Gamma(u)\varphi^2\, d\mu.\end{equation}
Moreover, from~$(1)$ we know that~$
\Gamma(\varphi)$ is a bounded function and therefore
\begin{equation}\label{9696}
\lim_{\varepsilon\to0}
\int_{X} \left(\Gamma(u)+\varepsilon\right)\Gamma(\varphi)\, d\mu
=
\int_{X} \Gamma(u)\Gamma(\varphi)\, d\mu
+
\lim_{\varepsilon\to0}
\varepsilon
\int_{X} \Gamma(\varphi)\, d\mu
=\int_{X} \Gamma(u)\Gamma(\varphi)\, d\mu.
\end{equation}
We also remark that, for any~$f,g\in{\mathcal{A}}$,
\begin{equation}\label{100}\begin{split}
\Gamma(f,g+\varepsilon)\;&=\;
\frac{1}{2}\left[L(f(g+\varepsilon))-fL(g+\varepsilon)-(g+\varepsilon)
Lf\right]\\
&=\;
\frac{1}{2}\left[L(fg)+\varepsilon Lf-
fLg-\varepsilon fL(1)-gLf-\varepsilon Lf\right]
\\ &=\;
\frac{1}{2}\left[L(fg)-
fLg-gLf\right]\\
&=\;\Gamma(f,g),
\end{split}\end{equation}
thanks to~\eqref{f per g}
and~\eqref{and sp}. As a consequence, we have that
\begin{equation}\label{9797}
\int_{X}\Gamma(\varphi^2, \Gamma(u)+\varepsilon)\, d\mu
=
\int_{X}\Gamma(\varphi^2, \Gamma(u))\, d\mu.\end{equation}
Furthermore, we claim that
\begin{equation}\label{9898}
\limsup_{\varepsilon\to0}
\int_{X}
\varphi^2 \Gamma\left(\sqrt{\Gamma(u)+\varepsilon}\right)\, d\mu
\le
\int_{X}
\varphi^2 \Gamma\left(\sqrt{\Gamma(u)}\right)\, d\mu
.\end{equation}
To prove this, we can assume that
\begin{equation}\label{9999}
\int_{X}
\varphi^2 \Gamma\left(\sqrt{\Gamma(u)}\right)\, d\mu
<+\infty,\end{equation}
otherwise~\eqref{9898} is true by default. Also, in view of~\eqref{9292}
(used here with~$f=g=\sqrt{\Gamma(u)+\varepsilon}$), we see that
$$ 
\Gamma(\Gamma(u)+\varepsilon)=4(\Gamma(u)+\varepsilon)
\,\Gamma\left(\sqrt{\Gamma(u)+\varepsilon}\right).$$
{F}rom this and~\eqref{100}, we obtain that
\begin{equation}\label{101}
\Gamma\left(\sqrt{\Gamma(u)+\varepsilon}\right)=\frac{
\Gamma(\Gamma(u)+\varepsilon)}{4(\Gamma(u)+\varepsilon)}
=
\frac{
\Gamma(\Gamma(u))}{4(\Gamma(u)+\varepsilon)}
.\end{equation}
Similarly, using~\eqref{9292}
with~$f=g=\sqrt{\Gamma(u)}$, we see that
\begin{equation}\label{sei pr}
\Gamma(\Gamma(u))=4(\Gamma(u))
\,\Gamma\left(\sqrt{\Gamma(u)}\right).\end{equation}
Inserting this into~\eqref{101}, we conclude that
$$ \varphi^2\Gamma\left(\sqrt{\Gamma(u)+\varepsilon}\right)=
\varphi^2\,\frac{\Gamma(u)}{\Gamma(u)+\varepsilon}
\,\Gamma\left(\sqrt{\Gamma(u)}\right)\le
\varphi^2\,\Gamma\left(\sqrt{\Gamma(u)}\right),$$
and the latter is a summable function, thanks to~\eqref{9999}.
Therefore, 
\begin{eqnarray*}
\int_{X}
\varphi^2 \Gamma\left(\sqrt{\Gamma(u)+\varepsilon}\right)\, d\mu
&=&
\int_{X}
\varphi^2\,\frac{\Gamma(u)}{\Gamma(u)+\varepsilon}
\,\Gamma\left(\sqrt{\Gamma(u)}\right)
\, d\mu
\\ &=&
\int_{X}
\varphi^2\,\left(1-\frac{\varepsilon}{\Gamma(u)+\varepsilon}
\right)\,\Gamma\left(\sqrt{\Gamma(u)}\right)
\, d\mu
\\
&\le&
\int_{X}
\varphi^2\,\Gamma\left(\sqrt{\Gamma(u)}\right)
\, d\mu,
\end{eqnarray*}
which in turn implies~\eqref{9898}.

Therefore, letting $\varepsilon \to 0$ in~\eqref{9494},
and exploiting~\eqref{9595}, \eqref{9696}, \eqref{9797}
and~\eqref{9898},
we conclude that
\begin{equation*}
\int_{X}F'(u)\Gamma(u)\varphi^2\, d\mu \leq \int_{X}\Gamma(u)\Gamma(\varphi)+ \frac{1}{2}\Gamma(\varphi^2, \Gamma(u))+\varphi^2 \Gamma\left(\sqrt{\Gamma(u)}\right)\, d\mu.
\end{equation*}
As a consequence, by $(6')$, Definition \ref{defGamma2}
(used here with~$f=g=u$) and~\eqref{eqncomp}, we obtain that
\begin{equation}\label{102}
\begin{split}
\int_{X}F'(u)\Gamma(u)\varphi^2\, d\mu &\leq \int_{X} \Gamma(u)\Gamma(\varphi)- \frac{1}{2}\,\varphi^2 \,L\Gamma(u)+ \varphi^2\,\Gamma\left(\Gamma(u)^{\frac{1}{2}}\right)\, d\mu\\
&=\int_{X} \Gamma(u)\Gamma(\varphi)-\varphi^2\, 
\Big(\Gamma_2(u)+\Gamma(u, Lu)\Big) +\varphi^2 \Gamma\left(\Gamma(u)^{\frac{1}{2}}\right)\, d\mu\\
&=\int_{X} \Gamma(u)\Gamma(\varphi)- \varphi^2 \,\Big(\Gamma_2(u)+\Gamma(u, -F(u))\Big)+\varphi^2 \Gamma\left(\Gamma(u)^{\frac{1}{2}}\right)\, d\mu.
\end{split}\end{equation}
Besides, using~\eqref{9090} with~$\Psi=-F$, we see that
$$  \Gamma(u,-F(u))
\Gamma(-F(u),u)=
-F'(u)\,\Gamma(u)\qquad \mu-\mbox{a.e.}
$$
and thus~\eqref{102} becomes
$$\int_{X}F'(u)\Gamma(u)\varphi^2\, d\mu \le
\int_{X} \Gamma(u)\Gamma(\varphi)- \Gamma_2(u)\varphi^2+F'(u)\Gamma(u)\varphi^2 +\varphi^2 \Gamma\left(\Gamma(u)^{\frac{1}{2}}\right)\, d\mu.
$$
Then, canceling one term, we obtain~\eqref{GF}, as desired.
\end{proof}
 
With this, we are able to prove Theorem~\ref{main1}:

\begin{proof}[Proof of Theorem~\ref{main1}:]

Using 
the identity in~\eqref{sei pr} and
Theorem~\ref{QUI0} \label{QUI}
we have that
\begin{equation}\label{11-039}
\begin{split}&
4\Gamma(u)\,\Big(
\Gamma_2(u)-\Gamma\big(\Gamma(u)^{\frac{1}{2}}\big)\Big)
=
4\Gamma(u)\,
\Gamma_2(u)-
4\Gamma(u)\,
\Gamma\big(\Gamma(u)^{\frac{1}{2}}\big)
\\ &\qquad= 4\Gamma(u)\,
\Gamma_2(u)-\Gamma(\Gamma(u))
\ge 4K\,(\Gamma(u))^2,
\end{split}\end{equation}
which is always nonnegative if~$K\ge0$.
Also, we know that~$\Gamma(0)=\Gamma_2(0)=0$, due to~\eqref{f per g}
and~\eqref{GammaD2}, and therefore
$$ \Gamma_2(u)-\Gamma\big(\Gamma(u)^{\frac{1}{2}}\big)=0
\qquad \mbox{in}\ \{\Gamma(u)=0\}.$$
Using this and Theorem \ref{eqnlin} 
we can write
\begin{equation}\begin{split}\label{stimaGamma2}
\int_{\{ \Gamma(u)\neq 0\}}\left(\Gamma_2(u)-\Gamma\left(\Gamma(u)^{\frac{1}{2}}\right)\right)\varphi^2\, d\mu
=\;&
\int_{X}\left(\Gamma_2(u)-\Gamma\left(\Gamma(u)^{\frac{1}{2}}\right)\right)\varphi^2\, d\mu
\\ \leq \;&\int_{X} \Gamma(u)\Gamma(\varphi)\, d\mu
\end{split}\end{equation}
for every $\varphi\in\mathcal{A}_0$.

Now, we fix~$R>1$ and define~$\Phi=\Phi_R\in C^{\infty}(\bbR)$,
with
\begin{equation}\label{test2}
|\Phi'(t)|\leq 3 
\end{equation}
for any $|t|\in [R, R+1]$ and
\begin{equation}\label{test}
\Phi(t):=\left\{\begin{array}{lll}
1 & \mbox{if}& |t|\leq R,\\
0 & \mbox{if}& |t|\geq R+1.
\end{array}
\right.
\end{equation}
We also define
\begin{equation}\label{WI}\widetilde{\Phi}(t):=\Phi(t)-1,\end{equation} for any~$t\in\R$.
In this way, we have that~$\widetilde{\Phi}(0)=0$. 

Moreover, we consider~$\xi_k$ as given by~(5). 
We remark that, in light of~\eqref{9292} and~\eqref{xi2},
\begin{equation*}
\Gamma(\xi_k^2)=4\xi_k^2\Gamma(\xi_k)\le\frac{4\xi_k^2}{k}.\end{equation*}
Iterating this, we have that
\begin{equation}\label{xi2-tris}
\Gamma(\xi_k^4)=4\xi_k^4\Gamma(\xi_k^2)\le\frac{16\xi_k^4}{k}.\end{equation}
Therefore, by possibly renaming~$\xi_k$ into~$\xi_k^4$,
we can suppose in view of~(5) that~$\xi_k(x)\in[0,1]$ for any~$x\in X$, and
\begin{equation}\label{inse}
\lim_{k\to+\infty} \xi_k=\mathrm{1}\qquad \mu-\mbox{a.e. in $X$}
\end{equation}
and
\begin{equation}\label{xi2-bis}
\Gamma(\xi_k)\le\frac{16\xi_k^4}{k}.\end{equation}
Then,
using (1') and the setting in~\eqref{DIS2}, we obtain that for every $k\in\mathbb{N}$
the function $\rho_k\xi_k$ belongs to $\mathcal{A}_0$.
Since $\mathcal{A}_0$ is a vector space we conclude
that
\begin{equation}\label{WI2}
 \varphi_k:=\widetilde{\Phi}(\rho_k\xi_k)+\xi_k\in \mathcal{A}_0.\end{equation}
Also, exploiting the bilinearity of $\Gamma$ we get
\begin{equation}\label{zero99}
\Gamma(\varphi_k)=\Gamma\big(\widetilde{\Phi}(\rho_k\xi_k)+\xi_k\big)
=\Gamma\big(\widetilde{\Phi}(\rho_k\xi_k)\big)
+2\Gamma\big(\widetilde{\Phi}(\rho_k\xi_k),\xi_k\big)+\Gamma(\xi_k).
\end{equation}
Moreover, using formula~\eqref{9090} with~$\Psi:=\widetilde{\Phi}$,
we obtain that
\begin{equation}\begin{split}\label{1}
\Gamma(\widetilde{\Phi}(\rho_k\xi_k))&=(\widetilde{\Phi}'(\rho_k\xi_k))^2\Gamma(\rho_k\xi_k)\\
&=(\Phi'(\rho_k\xi_k))^2\Big(\xi_k^2\Gamma(\rho_k)+
2\rho_k\xi_k\Gamma(\rho_k,\xi_k)+\rho_k^2\Gamma(\xi_k)\Big)
\end{split}\end{equation}
and
\begin{equation}\begin{split}\label{2}
\Gamma\big(\widetilde{\Phi}(\rho_k\xi_k),\xi_k\big)&=\Phi'(\rho_k\xi_k)\Gamma(\rho_k\xi_k, \xi_k)\\
&=\Phi'(\rho_k\xi_k)\Big(\xi_k\Gamma(\rho_k,\xi_k)+\rho_k\,\Gamma(\xi_k)\Big).
\end{split}\end{equation}
Plugging~\eqref{1} and~\eqref{2} into~\eqref{zero99}, we have that
\begin{equation}\begin{split}\label{tre}
\Gamma(\varphi_k)
=\;&(\Phi'(\rho_k\xi_k))^2\Big(\xi_k^2\Gamma(\rho_k)+
2\rho_k\xi_k\Gamma(\rho_k,\xi_k)+\rho_k^2\Gamma(\xi_k)\Big)\\&\qquad
+ 2\Phi'(\rho_k\xi_k)\Big(\xi_k\Gamma(\rho_k,\xi_k)+\rho_k\Gamma(\xi_k)\Big)
+\Gamma(\xi_k).
\end{split}\end{equation}
Now, taking~$\varphi:=\varphi_k$ into~\eqref{stimaGamma2}
and making use of~\eqref{tre}, we get that
\begin{equation}\begin{split}\label{ineqk:PRE}
& \int_{\{ \Gamma(u)\neq 0\}}\left(\Gamma_2(u)
-\Gamma\left(\Gamma(u)^{\frac{1}{2}}\right)\right)\varphi_k^2\, d\mu\\
\leq\;& \int_{X} \Gamma(u)\Big[(\Phi'(\rho_k\xi_k))^2\Big(
\xi_k^2\Gamma(\rho_k)+2\rho_k\xi_k\Gamma(\rho_k,\xi_k)+
\rho_k^2\Gamma(\xi_k)\Big)\\  &\qquad\qquad+ \Phi'(\rho_k\xi_k)
\Big(\xi_k\Gamma(\rho_k,\xi_k)+\rho_k\,\Gamma(\xi_k)\Big)+
\Gamma(\xi_k)\Big]\, d\mu.
\end{split}\end{equation}
We also point out that, in view of~\eqref{test2}
and~\eqref{test},
and recalling~\eqref{xi2-bis} and~\eqref{DIS2},
\begin{equation}
\begin{split}\label{0q3481-1-P}
& \int_{X} \Gamma(u)\Big[(\Phi'(\rho_k\xi_k))^2\Big(
\xi_k^2\Gamma(\rho_k)+2\rho_k\xi_k\Gamma(\rho_k,\xi_k)+
\rho_k^2\Gamma(\xi_k)\Big)\\  &\qquad\qquad+ \Phi'(\rho_k\xi_k)
\Big(\xi_k\Gamma(\rho_k,\xi_k)+\rho_k\,\Gamma(\xi_k)\Big)+
\Gamma(\xi_k)\Big]\, d\mu\\
\le\;&
\int_{X_{R,k}} \Gamma(u)\left[9\left(
C_0+2\rho_k\xi_k\,|\Gamma(\rho_k,\xi_k)|+
\frac{16\rho_k^2\xi_k^2}{k}\right)+ 3
\left(|\Gamma(\rho_k,\xi_k)|+\frac{16\rho_k\xi_k}{k}\right)\right]\,d\mu\\
&\qquad+\frac1k\int_X \Gamma(u)\,d\mu,
\end{split}\end{equation}
where
\begin{equation}\label{0138400}
X_{R,k} := \{ x\in X{\mbox{ s.t. }} \rho_k(x)\,\xi_k(x)\in[R,R+1]\}.\end{equation}
{F}rom~\eqref{0q3481-1-P} and~\eqref{0138400}, we have
\begin{equation}
\begin{split}\label{0q3481-1}
& \int_{X} \Gamma(u)\Big[(\Phi'(\rho_k\xi_k))^2\Big(
\xi_k^2\Gamma(\rho_k)+2\rho_k\xi_k\Gamma(\rho_k,\xi_k)+
\rho_k^2\Gamma(\xi_k)\Big)\\  &\qquad\qquad+ \Phi'(\rho_k\xi_k)
\Big(\xi_k\Gamma(\rho_k,\xi_k)+\rho_k\Gamma(\xi_k)\Big)+
\Gamma(\xi_k)\Big]\, d\mu\\
\le\;&
\int_{X_{R,k}} \Gamma(u)\left[9\left(
C_0+2(R+1)\,|\Gamma(\rho_k,\xi_k)|+
\frac{16(R+1)^2}{k}\right)+ 3
\left(|\Gamma(\rho_k,\xi_k)|+\frac{16(R+1)}{k}\right)\right]\,d\mu\\
&\qquad+\frac1k\int_X \Gamma(u)\,d\mu.
\end{split}\end{equation}
Now, we observe that
the following
Cauchy-Schwarz inequality for~$\Gamma$ holds true: given~$f$, $g\in{\mathcal{A}}$
and~$\alpha>0$, the fact that~$\Gamma(\cdot,\cdot)$
is a symmetric bilinear map implies that
$$ 0\le \Gamma\left(\alpha f\pm\frac1\alpha g\right)=
\Gamma\left(\alpha f\right)+\Gamma\left(\frac1\alpha g\right)
\pm2\Gamma\left(\alpha f,\frac1\alpha g\right)
=\alpha^2
\Gamma(f)+\frac1{\alpha^2}\Gamma(g)
\pm2\Gamma\left(f, g\right)
$$
and therefore
$$ 2\left|\Gamma\left( f, g\right)\right|\le \alpha^2
\Gamma(f)+\frac1{\alpha^2}\Gamma(g).$$
Then, by choosing $\alpha$ appropriately, we obtain that
$$ \left|\Gamma\left( f, g\right)\right|\le \big(\Gamma(f)\big)^{\frac12}\big(
\Gamma(g)\big)^{\frac12}.$$
As a consequence, we infer from~\eqref{0q3481-1} that
\begin{eqnarray*}
&& \int_{X} \Gamma(u)\Big[(\Phi'(\rho_k\xi_k))^2\Big(
\xi_k^2\Gamma(\rho_k)+2\rho_k\xi_k\Gamma(\rho_k,\xi_k)+
\rho_k^2\Gamma(\xi_k)\Big)\\  &&\qquad\qquad+ \Phi'(\rho_k\xi_k)
\Big(\xi_k\Gamma(\rho_k,\xi_k)+\rho_k\Gamma(\xi_k)\Big)+
\Gamma(\xi_k)\Big]\, d\mu\\
&\le&
\int_{X_{R,k}} \Gamma(u)\left[9\left(
C_0+\frac{32(R+1)\sqrt{C_0}}{\sqrt{k}}+
\frac{16(R+1)^2}{k}\right)+ 3
\left(\sqrt{ \frac{C_0}{{k}} }+\frac{16(R+1)}{k}\right)\right]\,d\mu\\
&&\qquad+\frac1k\int_X \Gamma(u)\,d\mu.
\end{eqnarray*}
We insert this information into~\eqref{ineqk:PRE}
and we take the limit as~$k\to+\infty$: 
in this way, we obtain that
\begin{equation}\label{S-12:01}
\begin{split}
&
\lim_{k\to+\infty} \int_{\{ \Gamma(u)\neq 0\}}\left(\Gamma_2(u)
-\Gamma\left(\Gamma(u)^{\frac{1}{2}}\right)\right)\varphi_k^2\, d\mu
\\ \le\;&\lim_{k\to+\infty}
\int_{X_{R,k}} \Gamma(u)\left[9\left(
C_0+\frac{32(R+1)\sqrt{C_0}}{\sqrt{k}}+
\frac{16(R+1)^2}{k}\right)+ 3
\left(\sqrt{\frac{C_0}{{k}}}+\frac{16(R+1)}{k}\right)\right]\,d\mu\\
&\qquad+\frac1k\int_X \Gamma(u)\,d\mu
\\ =\;&
9C_0\int_{X_{R}} \Gamma(u),
\end{split}
\end{equation}
where we have used~\eqref{inse}, \eqref{mai us} and~\eqref{0138400}
in the last step, and
$$ X_{R} := \{ x\in X{\mbox{ s.t. }} \rho(x)\in[R,R+1]\}.$$
We also notice that, by~\eqref{WI}, \eqref{inse}
and~\eqref{WI2},
$$ \lim_{k\to+\infty}
\varphi_k=\widetilde{\Phi}(\rho)+1=\Phi(\rho),$$
and therefore we deduce from~\eqref{S-12:01} that
\begin{equation}\label{S-12:02}
\int_{\{ \Gamma(u)\neq 0\}}\left(\Gamma_2(u)
-\Gamma\left(\Gamma(u)^{\frac{1}{2}}\right)\right)\big(\Phi(\rho)\big)^2\, d\mu
\le
9C_0\int_{X_{R}} \Gamma(u).\end{equation}
Noticing that~$\Phi(\rho(x))=1$ for any~$x\in X$ for which~$\rho(x)\le R$, thanks to~\eqref{test},
and recalling~\eqref{mai us}, we can take the limit as~$R\to+\infty$
in~\eqref{S-12:02}, obtaining that
\begin{equation}\label{S-12:03}
\int_{\{ \Gamma(u)\neq 0\}}\left(\Gamma_2(u)
-\Gamma\left(\Gamma(u)^{\frac{1}{2}}\right)\right)\, d\mu
\le0.\end{equation}
Since the integrand in the left hand side of~\eqref{S-12:03} is nonnegative
(recall~\eqref{11-039}), this
gives~\eqref{KPOSI0}.

Now we suppose that~$K>0$. 
Then, by Definition~\ref{DEF2.3}, we have that $(X,\mu,\Gamma)$ 
also satisfies the $CD(K,\infty)$ condition with~$K=0$.
Then, from \eqref{KPOSI0}
and \eqref{11-039}, we infer that~$
\Gamma(u)=0$ $\mu-$a.e. in~$ X$. If in addition~$\Gamma(u)\in C^0(X)$,
we get $\Gamma(u)=0$ in~$ X$,
which in light of (7') gives~\eqref{KPOSI}.
\end{proof}

\section{Applications to vector fields satisfying the H\"ormander condition}\label{oqidwuwiefyetytyyt}

Here we take $\mathcal{A}_0:=\mathrm{Lip}_c(\mathbb{R}^n)$
and $\mathcal{A}:=\mathrm{Lip}(\mathbb{R}^n)$. 
Moreover, we let~$\eta\in\mathcal{A}$ and define
\[
d\mu:=e^{\eta} dx.
\]
Let $Z_1,\ldots, Z_m$ be smooth vector fields in $\mathbb{R}^n$
with $$Z_j=\sum_{i=1}^n Z_i^j\partial_i.$$
We define
\begin{equation}\label{ewiegrbgv567u6576u}
Z_0 f:= \sum_{j=1}^m Z_j\eta\,Z_jf,
\qquad 
\mathrm{div}{Z_j}:=\sum_{i=1}^n \partial_i Z_i^j\quad{\mbox{ and }}\quad
\Delta_{Z}:=\sum_{j=1}^m Z_jZ_j.\end{equation}
We assume that the family $(Z_1,\ldots, Z_m)$ satisfies
the H\"ormander condition: at any point~$x\in\mathbb{R}^n$,
consider the vector spaces $V_p$ generated by the vector fields~$Z_j$
at~$x$, namely
\begin{align*}
&V_1:=\mathrm{span}\{Z_j\ |\ 1\leq j\leq m\},\\
&V_2:=\mathrm{span}\{Z_j, [Z_j,Z_k]\ |\ 0\leq j,k\leq m\},\\
&\cdots\\
&V_d:=\mathrm{span}\{V_{d-1}\cup \{[Z_j,V]\ |\ V\in V_{d-1},\ 
0\leq j\leq m\},
\end{align*}
then there exists $d\in\mathbb{N}$ such that,
for any~$x\in\mathbb{R}^n$, $V_d=\mathbb{R}^n$.

We recall the following result from~\cite{GN}:

\begin{teo}\label{Lip}
Let $Z=(Z_1,\ldots, Z_m)$ be a family of smooth vector fields
in~$\mathbb{R}^n$ satisfying the H\"ormander condition and
let~$d$ be the associated Carnot-Carath\'eodory distance~\cite{gromov},
which we assume to be continuous with respect to the
Euclidean topology.
If~$f: \mathbb{R}^n\to \mathbb{R}$ is a function such
that, for some~$\Lambda\geq 0$,
\begin{equation}\label{sss3}
|f(x)-f(y)|\leq \Lambda d(x,y)\qquad \mbox{for all}\ x,y\in\mathbb{R}^n,
\end{equation}
then the derivatives $Z_j f$, with~$j=1,\ldots, m$,
exist in distributional sense,
are measurable functions and~$|(Z f)(x)|\leq\Lambda$
for a.e. $x\in\mathbb{R}^n$.
\end{teo}

With the notation introduced in~\eqref{ewiegrbgv567u6576u},
we consider the linear operator
\begin{equation}\label{LHO}
Lg=\sum_{j=1}^m Z_j g\, \mathrm{div}{Z_j} +\Delta_{Z} g
+ Z_0 g.
\end{equation}
We remark that
\begin{equation}\label{HOR1}
Z_j(fg)=\sum_{i=1}^m Z_i^j \partial_i(fg)=
\sum_{i=1}^m Z_i^j f \partial_i g+\sum_{i=1}^m Z_i^j g\partial_i f=
fZ_j g+gZ_j f.\end{equation}
Therefore
\begin{equation*}
Z_jZ_j(fg) = Z_j (fZ_j g+gZ_j)=f Z_jZ_j g+gZ_jZ_j f+2Z_j fZ_jg,
\end{equation*}
that is, recalling~\eqref{ewiegrbgv567u6576u}, 
\begin{equation}\label{HOR2}
\Delta_Z(fg) = f \Delta_Z g + g\Delta_Z f+2\sum_{j=1}^m Z_j fZ_jg.
\end{equation}
Moreover, \eqref{HOR1} implies that
\begin{equation}\label{ngrhnrhth54865u8}
Z_0 (fg)= fZ_0 g + gZ_0f.
\end{equation}
In view of~\eqref{LHO},
\eqref{HOR1}, \eqref{HOR2} and~\eqref{ngrhnrhth54865u8},
and recalling~(4'), we define
\begin{align}\label{Gam}
\Gamma(f,g):=\sum_{j=1}^m Z_j f Z_j g.
\end{align}
Notice that
$$\Gamma(f)=\sum_{j=1}^m (Z_j f)^2= |Zf|^2.$$

We can now consider the diffusion Markov
triple~$(\mathbb{R}^n,\mu,\Gamma)$.
We observe that, with this choice, properties (1),(2) and (3) are satisfied. 
\smallskip

We now prove that condition~(4) is also satisfied.
For this, 
using an integration by parts, we point out that,
for any~$f\in C^\infty(\R^n)$ and any~$\varphi\in C^\infty_c(\R^n)$,
\begin{eqnarray*}
&& \int_{\R^n} Z_j f\,\varphi\,d\mu =\sum_{i=1}^n \int_{\R^n} Z_i^j \,\partial_i f\,\varphi\,e^\eta\,dx
=-\sum_{i=1}^n \int_{\R^n} f\, \partial_i(Z_i^j \,\varphi\,e^\eta)\,dx \\
&&\qquad =-\sum_{i=1}^n \int_{\R^n} f\, (\partial_iZ_i^j \,\varphi\,e^\eta
+Z_i^j \,\partial_i\varphi\,e^\eta+Z_i^j \,\varphi\,\partial_i\eta\,e^\eta)\,dx \\
&&\qquad=
-\int_{\R^n} f\, (\varphi\,{\rm div}Z_j
+Z_j \varphi+Z_j\eta \,\varphi)\,d\mu . 
\end{eqnarray*}
Taking~$\varphi:=Z_j g$, we thereby conclude that,
for every $f\in C^{\infty}(\mathbb{R}^n)$
and $g\in C_c^{\infty}(\mathbb{R}^n)$ (and, more generally,
for every~$f\in\mathrm{Lip}(\mathbb{R}^n)$
and~$g\in\mathrm{Lip}_c(\mathbb{R}^n)$ by a density argument),
it holds that
\begin{eqnarray*}
\int_{\mathbb{R}^n} \Gamma(f,g)\ d\mu &=&\sum_{j=1}^m
\int_{\mathbb{R}^n}  Z_j f \,Z_j g\,d\mu\\&=&
-\int_{\mathbb{R}^n} f\left(\sum_{j=1}^m Z_j g\,
\mathrm{div}{Z_j} +\Delta_{Z} g+Z_0g\right) \,d\mu\\
&=&-\int_{\mathbb{R}^n} f\,Lg\,d\mu ,
\end{eqnarray*}
which is~(4) in this setting.
\smallskip

We now prove condition~(5). We denote by $d(x)=d(x,0)$, and we consider
a function~$\Phi\in C^{\infty}(\bbR, [0,1])$,
with~$|\Phi'(t)|\leq 1$ for any~$|t|\in [1/8, 1/4]$, and
\begin{equation*}
\Phi(t):=\left\{\begin{array}{lll}
1 & \mbox{if}& |t|\leq 1/8,\\
0 & \mbox{if}& |t|\geq 1/4.
\end{array}
\right.
\end{equation*}
For every $k\in\mathbb{N}$, we define~$
\xi_k(x):=\Phi(d(x)^2/k^2)\in \mathrm{Lip}_c(\mathbb{R}^n)$.
Then $(\xi_k)_{k\in\mathbb{N}}$ is an increasing sequence with~$
\xi_k(x)\in [0,1]$ and~$\xi_k(x)\to 1$ as~$k\to+\infty$
for every $x\in\mathbb{R}^n$.
Moreover, we observe that~\eqref{sss3} in Theorem~\ref{Lip}
is satisfied taking~$f:=d$ with~$\Lambda=1$, and therefore we have 
that~$|Z d|\le 1$.
As a consequence
$$ |\Gamma(\xi_k)|= |Z\xi_k| = \frac{2|\Phi'(d^2/k^2)|}{k^2}d
|Zd|\leq \frac{1}{k},
$$
which completes the proof of~(5).
\smallskip

Notice that also (1'), (2'), (3'), (4'), (5') and (6') easily follow
from the very definition of~$\mathcal{A}$, $L$ and~$\Gamma$.

We claim that also (7') holds. Indeed, if~$f\in{\mathcal{A}}$
with~$\Gamma(f)=0$, then, by definition,
$$
\sum_{j=1}^m (Z_j f)^2 =0, $$
and so we have that~$Z_j f=0$ for any~$ j=1,\ldots, m$.
Thus, also all iterated derivatives vanish.
Then, the conclusion follows, since, by the H\"ormander condition,
every $\partial_{x_i}f$ can be written as a linear combination
of iterated derivatives. Therefore, $f$ is a constant function,
and so condition~(7') is satisfied.

\medskip

We now describe some interesting applications of the setting introduced above. \\

\subsection{Riemannian Manifolds} Let $(M,g)$ be a connected
Riemannian manifold of dimension $n$ equipped with the standard Levi-Civita connection $\nabla$ and let $G\in C^2(M)$. 
As customary in Riemannian geometry,
we define the gradient vector of~$f\in C^{\infty}(M)$
as the vector field whose coordinates are
\[
\nabla^i f=g^{ij}\partial_{x_j} f, \quad {\mbox{ for any }}\,1\leq i\leq n,
\]
where $g^{ij}$ are the coefficients of the inverse
matrix~$(g_{ij})_{1\leq i,j\leq n}$, and the repeated indices
notation has been used.
We consider the Markov triple $(M, \mu,\Gamma)$, where 
\[
\Gamma(f,g):=\nabla^i f\partial_{x_i} g.
\]
and 
\[
\mu:=e^{-G} dV.
\]
Here $dV$ denotes the Riemannian volume element,
namely, in local coordinates,
\begin{equation*}
dV = \sqrt{|g|}\,dx^1\wedge \dots \wedge dx^m,
\end{equation*}
where~$\{ dx^1,\dots,dx^n\}$ is the basis of~$1$-forms
dual to the vector basis~$\{\partial_1,\dots,\partial_n\}$.
The Laplace-Beltrami operator~$\Delta_g$ is defined on~$
f\in C^{\infty}(M)$ as
\[
\Delta_g f=\frac{1}{\sqrt{|g|}}\partial_{x_i}\left(\sqrt{|g|}g^{ij}\partial_{x_j}
f\right),
\]
whereas the Hessian matrix $\nabla^2 f$ of
a smooth function~$f$ is defined as
the symmetric~$2$-tensor given in a local patch by
\begin{equation*}
(\nabla^2 f)_{ij}=\partial^2_{ij}f-\Gamma^k_{ij}\partial_k f,\end{equation*}
where~$\Gamma^k_{ij}$ are the  
Christoffel symbols, namely
$$ 
\Gamma_{ij}^k=\frac12 g^{hk} \left( \partial_i g_{hj} +\partial_j g_{ih} -\partial_h g_{ij} \right) .$$
Given a tensor~$A$,
we also define its norm by~$|A|=\sqrt{A A^*}$, being~$A^*$
the adjoint of~$A$. 
\smallskip

As proved in \cite[1.11.10]{BLG} and \cite[1.16.4]{BLG}, we see that
\begin{equation}\label{0swo:1}
Lf=\Delta_gf-\left\langle \nabla G,\nabla f\right\rangle_g
\end{equation}
and
\begin{equation}\label{G GA2}
\Gamma_2(f)=|\nabla^2 f|+\mathrm{Ric}(L)(\nabla f,\nabla f),
\end{equation}
where $\mathrm{Ric}(L)$ is a symmetric tensor defined from the
Ricci tensor~$\mathrm{Ric}_g$ by
\[
\mathrm{Ric}(L)=\mathrm{Ric}_g+\nabla^2 G.
\]
Observing that
\[
\Gamma\left(\Gamma(f)^{\frac{1}{2}}\right)=|\nabla|\nabla f||^2,
\]
we use~\eqref{G GA2} and conclude that 
\begin{equation} \label{G GA3}
\Gamma_2(u)-\Gamma\left(\Gamma(u)^{\frac{1}{2}}\right)=
|\nabla^2 u|+\mathrm{Ric}(L)(\nabla u,\nabla u)-|\nabla|\nabla u||^2.
\end{equation}
Consequently,
for any stable weak solution~$u\in C^{\infty}(M)$ to 
\begin{equation}\label{EQM}
Lu+F(u)=0 \quad{\mbox{ in }}\; M,\end{equation} the
Poincar\'e inequality in~\eqref{GF} of Theorem~\ref{eqnlin}
reads as follows:
\begin{align}\label{GFRiem}
\int_M\Big(|\nabla^2 u|+\mathrm{Ric}(L)(\nabla u,\nabla u)
-|\nabla|\nabla u||^2\Big)\varphi^2 d\mu\leq \int_{M} |\nabla u|^2 
|\nabla \varphi|^2 d\mu,
\end{align}
for any~$\varphi\in C^{\infty}_c(M)$.

In particular, on~$\R^n$ equipped with the flat Euclidean metric,
and for the usual Laplacian~$\Delta$,
inequality \eqref{GFRiem} reads as
$$
\int_{\R^n}\Big(|\nabla^2 u|^2+\nabla^2 G(\nabla u,\nabla u)
-|\nabla|\nabla u||^2\Big)\varphi^2 d\mu\leq \int_{\R^n}
|\nabla u|^2|\nabla \varphi|^2 d\mu,
$$
for any~$\varphi\in C^{\infty}_c(\R^n)$,
which is precisely the inequality already proved
in~\cite{FarHab,CNP,CNV,FNP}.

Furthermore, if~$G:= 0$ then~\eqref{GFRiem} was proved
in~\cite{fsv1, fsv2, FMV}, whereas the general case seems to be
new in the literature.
\smallskip

In our setting,
we take~$\mathcal{A}_0:=\mathrm{Lip}_c(M)$
and~$\mathcal{A}:=\mathrm{Lip}(M)$,
where~$f\in\mathrm{Lip}(M)$ if
\begin{equation}\label{LIL0}
\sup_{x\ne y\in M}\frac{f(x)-f(y)}{d(x,y)}<+\infty,\end{equation}
being~$d$ the distance defined in~\eqref{DIS}.
Notice that, if we fix~$x_0\in M$ and we set~$
\rho(x):=d(x,x_0)$, it holds that
\begin{equation}\label{LIL}
\rho(x)-\rho(y)\le d(x,y).\end{equation}
Indeed, by~\eqref{DIS}, for any~$\eps>0$ there exists~$f_\eps$
with~$\Gamma(f_\eps)\le1$ such that~$\rho(x)\le \eps+ f_\eps(x)
-f_\eps(x_0)$,
and therefore
$$ \rho(x)-\rho(y)\le
\eps+ f_\eps(x)-f_\eps(x_0) - \big( f_\eps(y)-f_\eps(x_0)\big)
=\eps+f_\eps(x)-f_\eps(y)\le\eps+d(x,y).$$
{F}rom this, taking~$\eps$ as small as we wish, we obtain~\eqref{LIL}.

Then, comparing \eqref{LIL0} and \eqref{LIL}, we see that~$\rho\in
\mathrm{Lip}(M)=\mathcal{A}$. 
\smallskip

We also assume that
$$ \lambda\delta^{ij}\le g^{ij}\le \frac{ \delta^{ij} }\lambda,$$
for some~$\lambda\in(0,1]$. In this way, if~$|\cdot|_E$ is the Euclidean norm
of a vector, it holds that
$$ \lambda \,|v|_E^2 =\lambda \delta^{ij}v_i v_j\le g^{ij} v_i v_j.$$
Then, by~\eqref{DIS}, 
for any~$\eps>0$ there exists~$\tilde f_\eps$
with~$\Gamma(\tilde f_\eps)\le1$ such that
\begin{eqnarray*}
d(x,y) &\le&
\eps+ \tilde f_\eps(x)-\tilde f_\eps(y)\\
&=&
\eps+ \int_0^1 \partial_{x_i}\tilde f_\eps(x+ty)\cdot(x_i-y_i)\,dt\\
&\le&
\eps+ \frac1\lambda\,
\int_0^1 \sqrt{g^{ij}\partial_{x_i}\tilde f_\eps(x+ty)\,
\partial_{x_j}\tilde f_\eps(x+ty)}\,|x-y|_E\,dt
\\&=&
\eps+ \frac1\lambda\,
\int_0^1 \sqrt{\Gamma(\tilde f_\eps)(x+ty)}\,|x-y|_E\,dt
\\&=&
\eps+ \frac{|x-y|_E}\lambda.\end{eqnarray*}
Hence, taking~$\eps$ arbitrary small,
it follows that~$d(x,y)\le\frac{|x-y|_E}\lambda$.
Recalling~\eqref{LIL}, we thereby conclude that~$|\nabla \rho|$,
and therefore~$\Gamma(\rho)$, is bounded by a universal constant depending on~$\lambda$.
Consequently, the assumptions of Theorem~\ref{main1}
are satisfied. Then, using Theorem~\ref{main1}
and~\eqref{G GA3}, we obtain that, in this setting,
if the curvature dimension condition~$CD(K,\infty)$ holds true
for some~$K\geq 0$, and~$u$ is a stable solution
with~$\int_M |\nabla u|^2\,d\mu<+\infty$,
then:
\begin{align}
&K>0\Longrightarrow {\mbox{$u$ is constant in~$M$}};\\
\label{KPOSI0-B}
&K=0 \Longrightarrow
|\nabla^2 u|+\mathrm{Ric}(L)(\nabla u,\nabla u)-|\nabla|\nabla u||^2
=0\qquad\mu-\mbox{a.e. in }M.
\end{align}
To grasp a geometric flavor of~\eqref{KPOSI0-B}, one can fix a point~$p
\in M$ with~$\nabla u(p)\ne0$
and consider normal coordinates at~$p$ for which
\begin{equation}\label{NORMALC}
g^{ij}(p)=\delta^{ij},\qquad \partial_{x_k}g^{ij}(p)=0\qquad{\mbox{
and }}\qquad\Gamma_{ij}^k(p)=0,\end{equation}
see e.g. page~55 in~\cite{2009}.
Then, the level set~$S$ of~$ u$
passing through~$p$ is locally a submanifold of~$M$ of codimension~$1$,
endowed with a Riemannian structure induced by that of~$M$
(namely if~$v,w\in T_p S\subseteq T_p M$
one can consider~$g(v,w)$ as defining a metric on~$S$).
Consequently, in view of~\eqref{NORMALC},
we can reduce the Riemannian term~$
|\nabla^2 u|-|\nabla|\nabla u||^2$ to its Euclidean counterpart,
which, due to the classical Sternberg-Zumbrun identity
(see formula~(2.1) of~\cite{SZ2})
is larger than~$K^2\,|\nabla u|^2$,
being~$K^2$ the sum of the square of the eigenvalues
of the second fundamental form of~$S$, according to the induced Riemannian structure,
see  Proposition~18 in~\cite{FMV}.
Therefore, by~\eqref{KPOSI0-B},
if the Ricci tensor is nonnegative, it follows
that the second fundamental form of~$S$ at~$p$ vanishes,
and the Ricci tensor must vanish at~$p$ as well.

Submanifolds with vanishing second fundamental form
are called totally geodesic (see e.g. page~104 in~\cite{ONE}
or Proposition~1.2 in~\cite{RADE})
and are characterized by the property that
any geodesic on the submanifold is also a geodesic on the ambient manifold.

\subsection{Carnot groups} We recall that
a Carnot group~$\mathbb{G}$
is a connected Lie group whose Lie algebra~$\mathcal{G}$
is finite dimensional and stratified of step~$s\in\mathbb{N}$.
Precisely, there exist linear subspaces $V_1,\dots,V_s$ of $\mathcal{G}$ such that
\[\mathcal{G}=V_1\oplus \cdots \oplus V_s\]
with
\[[V_1,V_{i-1}]=V_{i}\;\; \mbox{ if }2\leq i\leq s  \quad
\mbox{ and }\quad [V_1,V_s]=\{0\}.\]
Here $[V_1,V_i]:=\mathrm{span}\{[a,b]: a\in V_1,\ b\in V_i\}.$
Since~$\mathcal{G}$ is stratified,
then every element of~$\mathcal{G}$ is the linear combination
of commutators of elements of~$V_1$.
We refer to~\cite{BLU} for a complete introduction to the subject.

Let $\mathrm{dim}(V_1)=m$ and $Z=(Z_1,\ldots, Z_m)$ be
a basis of~$V_1$.
The family~$Z$ satisfies the H\"ormander condition. Moreover, in
this setting,
\begin{equation}\label{feejiyh5}
\Gamma(f,g)=\sum_{i=1}^m Z_i f Z_i g\quad {\mbox{ and }}
\quad L f= \Delta_Z f.
\end{equation}
In order to compute $\Gamma_2$ we will use the
following Bochner-type
formula proved in~\cite[Proposition 3.3]{G} coupled
with~\cite[Lemma 3.1]{G}:

\begin{teo}\label{bochner}
Let $u$ be a smooth function. Then,
\begin{eqnarray*}
&&\frac12 \Delta_Z |Zu|^2\\
&=&\|Z^2 u \|^2+\sum_{j=1}^m Z_j u\, 
Z_j(\Delta_Zu)+ 2\sum_{i,j=1}^m 
Z_j u [Z_i,Z_j]Z_iu +\sum_{i,j=1}^m Z_j u [Z_i,[Z_i,Z_j]]u ,
\end{eqnarray*}
where $Z^2 u$ denotes the horizontal Hessian
matrix associated to the family~$Z$,
namely the~$m\times m$ matrix whose elements are given by~$
u_{ij}:=Z_iZ_j u$, with~$i,j=1,\dots, m$.
\end{teo}

Now, let $u\in C^{\infty}(\mathbb{G})$ be a stable solution
to
\begin{equation}\label{8768bnfkle}
\Delta_Z u+F(u)=0\quad {\mbox{ in }}\mathbb{G}.\end{equation}
Therefore, recalling~\eqref{feejiyh5} and using Theorem \ref{bochner},
\begin{align*}
&\Gamma_2(u) =\frac12L\Gamma (u) -\Gamma(u,Lu)=
\frac{1}{2}\Delta_Z |Zu|^2- \Gamma(u, \Delta_Z u)\\
&\quad
= \|Z^2 u \|^2+\sum_{j=1}^m Z_j u\, 
Z_j(\Delta_Zu)+ 2\sum_{i,j=1}^m 
Z_j u [Z_i,Z_j]Z_iu +\sum_{i,j=1}^m Z_j u [Z_i,[Z_i,Z_j]]u
- \Gamma(u, \Delta_Z u)\\
&\quad=
\|Z^2 u \|^2+ 2\sum_{i,j=1}^m 
Z_j u [Z_i,Z_j]Z_iu +\sum_{i,j=1}^m Z_j u [Z_i,[Z_i,Z_j]]u
\\&\quad
=\|Z^2 u \|^2+\mathcal{R}(u),
\end{align*}
where
\[
\mathcal{R}(u):=2\sum_{i,j=1}^m Z_j u [Z_i,Z_j]Z_iu +\sum_{i,j=1}^m Z_j u [Z_i,[Z_i,Z_j]]u.
\]
Therefore, if~$u\in C^{\infty}(\mathbb{G})$ is a
stable solution to~\eqref{8768bnfkle}, inequality \eqref{GF} reads as
\begin{align}\label{Carnot}
\int_{\mathbb{G}}\left(\|Z^2 u \|^2-|Z|Z u||^2
+\mathcal{R}(u)\right) \varphi^2\ dx\leq 
\int_{\mathbb{G}} |Z u|^2 |Z \varphi|^2\ dx,
\end{align}
for any~$\varphi\in C^{\infty}_c(\mathbb{G})$. 

Formula~\eqref{Carnot} generalizes to general Carnot groups
the Poincar\'e inequality obtained in~\cite{FV1} in the Heisenberg group
and in~\cite{PV} in the Engel group
(we refer the reader to~\cite{BLU} for the definitions,
and we remark that the divergence of the Heisenberg and Engel vector fields
vanish in the setting of~\eqref{ewiegrbgv567u6576u}).
In particular, in the case of the Heisenberg group,
formula~\eqref{Carnot} here reduces to formula~(7) in~\cite{FV1},
and in the case of the Engel group, formula~\eqref{Carnot} here reduces to the
formula in Proposition~3.7 of~\cite{PV} and~${\mathcal{R}}$ here coincides
with~${\mathcal{J}}$ in Theorem~1.1 of~\cite{PV}.
In its full generality, our formula~\eqref{Carnot} seems to be new in the literature.

In addition, the distance
in~\eqref{DIS} coincides with that of Carnot-Carath\'eodory 
in this setting, see~\cite{BLG}, and so~\eqref{DIS2} holds true in this case.
For completeness, we state~\eqref{Carnot} and we apply Theorem~\ref{main1}
to obtain this original result:
\begin{teo}
Let~$\mathbb{G}$ be
a Carnot group whose Lie algebra~$\mathcal{G}=
V_1\oplus \cdots \oplus V_s$ is stratified of step~$s$,
with~$V_1$ generated by the basis of vector fields~$(Z_1,\dots,Z_m)$
that satisfy the H\"ormander condition.
Let~$u\in C^\infty(\mathbb{G})$ be a
stable weak solution to~$\Delta_Zu+F(u)=0$
in~$\mathbb{G}$. Then,
$$
\int_{\mathbb{G}}\left(\|Z^2 u \|^2-|Z|Z u||^2
+\mathcal{R}(u)\right) \varphi^2\ dx\leq 
\int_{\mathbb{G}} |Z u|^2 |Z \varphi|^2\ dx,
$$for any~$\varphi\in C^{\infty}_c(\mathbb{G})$. 

Assume also that
\begin{eqnarray*}&& \mathcal{R}(u)\ge 0\\{\mbox{and }}&&
\int_{\mathbb{G}} |Zu|^2\, dx <\infty.\end{eqnarray*}
Then
\begin{eqnarray*}&& \mathcal{R}(u)= 0\\{\mbox{and }}&& \|Z^2 u\|^2=|Z|Z u||^2
\mbox{ a.e. in }\mathbb{G}.
\end{eqnarray*}
\end{teo}

Now, we prove that, for a particular family of Carnot groups,
formula~\eqref{Carnot} provides a geometric inequality
for every stable solution to~\eqref{8768bnfkle}.
Model filiform groups are the Carnot groups with the simplest
Lie brackets possible while still having arbitrarily large step,
see~\cite{Mon02}. They have previously been investigated in connection
with non-rigidity of Carnot groups \cite{O08}, quasiconformal mappings between
Carnot groups \cite{Xia05}, geometric control theory \cite{Mon02} and geometric measure theory \cite{PS}.

The formal definition is as follows:

\begin{defi}\label{filiform}
Let $n\geq 2$. The \emph{model filiform group of step $n-1$}
is the Carnot group~$\mathbb{E}_{n}$ whose Lie algebra~$
\mathcal{E}_{n}$ admits a basis~$Z_{1}, \ldots, Z_{n}$
satisfying~$[Z_{i},Z_{1}]=Z_{i+1}$ for~$1<i<n$,
with all other Lie brackets among the~$Z_{i}$ equal to zero. 

The stratification of $\mathcal{E}_{n}$ is
$$ \mathcal{E}_{n}=V_{1}\oplus \cdots \oplus V_{n-1}$$
with~$V_{1}=\mathrm{Span}\{Z_{1}, Z_{2}\}$ and $V_{i}=\mathrm{Span}\{Z_{i-1}\}$ for $1<i<n$.
\end{defi}

Proceeding exactly as in \cite[formula (19)]{FV1}, we get
\begin{equation*}
|Z|Z u||^2=\frac{1}{|Z u|^2}
\left\langle H_Zu \,Z u, Zu\right\rangle\quad
\mbox{in}\  \{Z u\neq 0\},
\end{equation*}
where
\[
H_Z u:=Z^2 u(Z^2 u)^T.
\]
Whenever $P\in \{u=k\}\cap \{Z u\neq 0\}$,
we can consider the smooth surface~$\{u=k\}$ and
define the intrinsic normal to~$\{u=k\}$ and the
intrinsic unit tangent direction to~$\{u=k\}$ as
\[
\nu:=\frac{Z u(P)}{|Z u(P)|}\quad \mbox{and}\quad v:=\frac{(Z_2 u(P), -Z_1u(P))}{|Z u(P)|},
\]
respectively.
We observe that \cite[Lemma 2.1]{FV1} only depends on
the fact that $\mathrm{dim} V_1=2$ and $\mathrm{dim} V_2=1$,
which still hold in every model filiform Carnot group.
Therefore, the following result holds:

\begin{lem}
On $\{u=k\}\cap \{Z u \neq 0\}$, it holds that
\begin{align}\label{ddw}
\|Z^2 u\|^2-\left\langle (H_Z u) \nu, \nu\right\rangle=|Z u|^2\left[h^2+ \left(p+\frac{\left\langle (Hu) v, \nu\right\rangle}{|Z u|}\right)^2\right]
\end{align}
where
\[
h=\mathrm{div}_Z \nu=Z_1 \nu_1+Z_2 \nu_2
\]
and
\[
p=-\frac{Z_3u}{|\nabla_Z u|}.
\]
\end{lem}

Plugging \eqref{ddw} into \eqref{Carnot} we get the following geometric Poincar\'e inequality:
\begin{align*}
\int_{\{Zu\neq 0\}}\left(|Z u|^2\left[h^2+ \left(p+\frac{\left\langle (Hu) v, \nu\right\rangle}{|Z u|}\right)^2\right]
+\mathcal{R}(u)\right) \varphi^2\ dx\leq 
\int_{\mathbb{E}_{n}} |Z u|^2 |Z \varphi|^2\ dx,
\end{align*}
for any~$\varphi\in C^{\infty}_c(\mathbb{E}_{n})$.

We summarize this statement and that of Theorem~\ref{main1}
in the following original result:

\begin{teo}\label{saq}
Let~$\mathbb{E}_{n}$ be a 
model filiform group of step $n-1$, as in Definition~\eqref{filiform}.
Let~$u\in C^\infty(\mathbb{E}_{n})$ be a
stable weak solution to~$\Delta_Zu+F(u)=0$
in~$\mathbb{E}_{n}$. Then,
\begin{align*}
\int_{\{Zu\neq 0\}}\left(
|Z u|^2\left[h^2+ \left(p+\frac{\left\langle (Hu) v, \nu\right\rangle}{|Z u|}\right)^2\right]
+\mathcal{R}(u)\right) \varphi^2\ dx\leq 
\int_{\mathbb{E}_{n}} |Z u|^2 |Z \varphi|^2\ dx,
\end{align*}
for any~$\varphi\in C^{\infty}_c(\mathbb{E}_{n})$. 

Assume also that
\begin{eqnarray*}&&
\mathcal{R}(u)
\ge 0\\{\mbox{and }}&&
\int_{\mathbb{E}_{n}} |Zu|^2\, dx <\infty.\end{eqnarray*}
Then:
\begin{eqnarray}&& \nonumber
\mathcal{R}(u)
= 0\quad{\mbox{and }}\\&&
\left.\begin{matrix}
h=0\\
\\
p+\displaystyle\frac{\left\langle (Hu) v, \nu\right\rangle}{|Z u|}
=0\quad\end{matrix}\right\}\quad
\mbox{ a.e. in } {\mathbb{E}_{n}}\cap\{ Zu\ne0\}.\label{09222133}
\end{eqnarray}
\end{teo}

The two equations in~\eqref{09222133} on~$ \{ Zu\ne0\}$
can be seen as ``intrinsic geodesic equations''
on the noncritical level sets of the solution~$u$. 

\subsection{Grushin plane}
For a given $\alpha\in \mathbb{N}$, the vector fields $Z_1=\partial_x$
and $Z_2=|x|^{\alpha}\partial_y$ satisfy the H\"ormander condition
in~$\mathbb{R}^2$. We call Grushin plane the metric space~$\mathbb{G}_{\alpha}
=(\mathbb{R}^2, d)$, where $d$ is the Carnot-Carath\'eodory
distance induced
by~$Z_1$ and~$Z_2$. Background on the
Grushin plane may be found in~\cite{Bel,Mon02}.

Since $\mathrm{div}Z_1=\mathrm{div}Z_2=0$ and $Z_0=0$, then,
from what we proved for vector fields satisfying the H\"ormander condition, we get
\[
Lu=\Delta_Z u,\quad  \Gamma(f,g)=Z_1 f Z_1 g+Z_2 fZ_2 g\quad
\mbox{and}\quad d\mu= dx.
\]
For every solution $u\in C^{\infty}(\mathbb{G}_{\alpha})$
to~$\Delta_Z u+ F(u)=0$, we see that, for any~$i\in\{1,2\}$,
\begin{equation}\label{9yhn83tugrgbe}
Z_i\Delta_Z u= -F'(u)\,Z_i u. \end{equation}
Let also $Z_3:=[Z_1,Z_2]$. We observe that,
by \cite[Lemma 2.1]{FV3} and~\eqref{9yhn83tugrgbe}, we have
\begin{align}\label{ddsa}
&\Delta_Z Z_1 u=Z_1\Delta_Z u- 2Z_3Z_2 u
= -F'(u)\,Z_1 u- 2Z_3Z_2 u
\\
\label{ddsa1} {\mbox{and }}\quad
&\Delta_Z Z_2 u=Z_2\Delta_Z u+ 2Z_3Z_1 u=-F'(u)\,Z_2 u+ 2Z_3Z_1 u.
\end{align}
As a further consequence of~\eqref{9yhn83tugrgbe}, we obtain that 
\begin{equation}\begin{split}\label{rekhtbfndb}
& \Gamma(u,\Delta_Z u)=
Z_1 u Z_1\Delta_Z u + Z_2 u Z_2\Delta_Z u\\&\qquad=
-F'(u)\left( |Z_1u|^2 +|Z_2u|^2\right)
=-F'(u)|Zu|^2 .
\end{split}\end{equation}
Moreover, direct calculations give
\begin{align*}
Z_1Z_1(Z_1u)^2&=2(Z_1Z_1u)^2+2Z_1 uZ_1Z_1Z_1 u\\
&=2(Z_1Z_1u)^2+2Z_1 u \Delta_Z Z_1 u-2Z_1u Z_2Z_2Z_1u,\\
Z_2Z_2(Z_1u)^2&=2(Z_2Z_1u)^2+2Z_1u Z_2Z_2Z_1 u\\
Z_1Z_1(Z_2u)^2&=2(Z_1Z_2u)^2+2Z_2u Z_1Z_1Z_2 u,\\
{\mbox{and }}\quad
Z_2Z_2(Z_2u)^2&=2(Z_2Z_2u)^2+2Z_2 uZ_2Z_2Z_2 u
\\&=2(Z_2Z_2u)^2+2Z_2 u \Delta_Z Z_2u -2Z_2u Z_1Z_1Z_2 u.
\end{align*}
Summing up these
equalities and recalling~\eqref{ddsa} and~\eqref{ddsa1},
we get
\begin{align*}
\Delta_Z|Zu|^2&=2\|Z^2 u\|+2Z_1 u \Delta_Z Z_1 u+2Z_2 u \Delta_Z Z_2u\\
&=2\|Z^2 u\|+2Z_1u (-F'(u)Z_1 u -2Z_3Z_2 u)+2Z_2 u(-F'(u)Z_2 u+ 2Z_3Z_1 u)\\
&=2\|Z^2 u\| -2F'(u)|Zu|^2 -4Z_1uZ_3Z_2 u+ 4Z_2uZ_3Z_1 u.
\end{align*}
Using this and~\eqref{rekhtbfndb}, we
are now in position to compute~$\Gamma_2(u)$ as follows:
\begin{align*}
\Gamma_2(u)&=\frac{1}{2}L\Gamma(u)-\Gamma(u, Lu)
=\frac{1}{2} \Delta_Z|Z u|^2-\Gamma(u,\Delta_Z u)\\
&=\|Z^2 u\| -2Z_1uZ_3Z_2 u+ 2Z_2uZ_3Z_1 u.
\end{align*}
Therefore, if~$u\in C^{\infty}(\mathbb{G}_{\alpha})$ is a
stable solution to $\Delta_Z u+ F(u)$, inequality \eqref{GF} reads as
\begin{align*}
\int_{\mathbb{G}_{\alpha}}\Big(\|Z^2 u \|^2-|Z|Z u||^2
-2Z_1uZ_3Z_2 u+ 2Z_2uZ_3Z_1 u\Big) \varphi^2\ dx\leq 
\int_{\mathbb{G}_{\alpha}} |Z u|^2 |Z \varphi|^2\ dx,
\end{align*}
for any~$\varphi\in C^{\infty}_c(\mathbb{G}_{\alpha})$, which coincides
with
formula~(1.10) in Theorem~1.1 of~\cite{FV3}.

We conclude pointing out that, since the
proof of Theorem~\ref{saq} is based on the commutator relations
of the vectors~$Z_1$ and~$Z_2$,
the same result also holds in the Grushin plane
(compare with~(1.15) and~(1.16) of~\cite{FV3}).

\vfill

\end{document}